\documentclass[11pt,epsf]{article}
\usepackage{graphicx}
\usepackage{amsthm}
\usepackage{amsfonts}
\usepackage{amssymb}
\title{The closest to $0$ spectral number of the partial theta function}
\author{Vladimir Petrov Kostov\\ 
Universit\'e de Nice, 
Laboratoire de Math\'ematiques, Parc Valrose,\\ 06108 Nice Cedex 2, France,  
e-mail: kostov@math.unice.fr}
\date{}
\newtheorem{tm}{Theorem}

\newtheorem{rem}[tm]{Remark}

\newtheorem{lm}[tm]{Lemma}

\newtheorem{cor}[tm]{Corollary}
\newtheorem{prop}[tm]{Proposition}

\begin{document} 
\maketitle 

\begin{abstract}
The {\em spectrum} of the partial theta function 
$\theta :=\sum _{j=0}^{\infty}q^{j(j+1)/2}x^j$ is the set of values of 
$q\in \mathbb{C}$, $0<|q|<1$, for which $\theta (q,.)$ has a multiple zero. 
We show that the only element of the spectrum belonging to the disk 
$\mathbb{D}_{0.31}$ is $0.3092493386\ldots$. 
\end{abstract}

\section{Introduction}

The {\em partial theta function} is the sum of the 
bivariate series $\theta (q,x):=\sum _{j=0}^{\infty}q^{j(j+1)/2}x^j$. 
For any fixed value of the parameter $q$ from the open unit disk it is 
an entire function in the complex variable $x$.  
(Recall that the Jacobi theta function is the sum of the series 
$\Theta (q,x):=\sum _{j=-\infty}^{\infty}q^{j^2}x^j$ whereas 
$\theta (q^2,x/q)=\sum _{j=0}^{\infty}q^{j^2}x^j$.) 
The function $\theta$ finds applications in the theory 
of (mock) modular forms (see \cite{BrFoRh}), in  Ramanujan type $q$-series 
(see \cite{Wa}), 
in asymptotic analysis (see \cite{BeKi}), in statistical physics 
and combinatorics (see \cite{So}), and also in questions concerning 
hyperbolic polynomials 
(i.e. real polynomials with all zeros real, see 
\cite{Ost}, \cite{KaLoVi}, 
\cite{KoSh} and \cite{Ko2}). These questions 
have been considered by Hardy, Petrovitch and Hutchinson 
(see \cite{Ha}, \cite{Hu} and \cite{Pe}). For more facts about $\theta$,  
see also~\cite{AnBe}.

The {\em spectrum} of $\theta$ (defined by B. Z. Shapiro in \cite{KoSh}) 
is the set of values of $q$ for which 
$\theta (q,.)$ has a multiple zero in $x$. Properties of $\theta$ and its 
spectrum are studied in \cite{Ko1}, \cite{Ko2}, \cite{Ko3}, \cite{Ko4}, 
\cite{Ko5}, \cite{Ko7}  
and~\cite{Ko8}. Thus if $q\in (0,1)$, $x\in \mathbb{R}$, for the 
corresponding spectral numbers $0<\tilde{q}_1<\tilde{q}_2<\cdots <1$ one has 
$\tilde{q}_j=1-\pi /2j+(\log j)/8j^2+O(1/j^2)$ 
and the respective double zeros of 
$\theta$ are of the form $y_j=-e^{\pi}e^{-(\log j)/4j+O(1/j)}$ (see \cite{Ko8}). 
If 
$q\in (-1,0)$, $x\in \mathbb{R}$, there are the spectral numbers 
$-1<\cdots <\bar{q}_{k+1}<\bar{q}_k<0$, $\bar{q}_k=1-\pi /8k+o(1/k)$ 
and the double zeros $\bar{y}_k$, $|\bar{y}_k|\rightarrow e^{\pi /2}$, 
see \cite{Ko7}. The spectral number 
$\tilde{q}=\tilde{q}_1:=0.3092493386\ldots$ 
is connected with {\em section-hyperbolic polynomials}, see~\cite{KoSh}. The 
function $\theta (\tilde{q},.)$ has a single double zero at 
$-7.5032559833\ldots$, its other zeros are real negative and simple.
We denote by $\mathbb{D}_a$ a closed disk centered at $0$ and of radius $a>0$ 
(in the $x$- or $q$-space). Subscripts indicate derivations 
(e.g. $\theta _x=\partial \theta /\partial x$). 
The zeros $\xi _j$ of $\theta$ can be expanded in Laurent series in 
$q$, $\xi _j=-1/q^j+O(q^{j(j-1)/2})$, convergent for 
$q\in \mathbb{D}_{\rho _j}\backslash 0$, 
$\rho _j\geq 0.108$, see \cite{Ko4} and~\cite{Ko6}; 
for $q\in \mathbb{D}_{0.108}$, 
$\theta (q,.)$ has no multiple zeros.

\begin{tm}\label{tmmain}
(1) In $\mathbb{D}_{0.31}$,  
$\tilde{q}$ is the only spectral 
number of $\theta$.

(2) For any disk $\mathbb{D}_a$, $a\in (0,1)$, there exists $b>0$ such that 
$\theta$ has no multiple zeros in 
$B_{a,b}:=\mathbb{D}_a\times (\mathbb{C}\backslash \mathbb{D}_b)$, it has only 
isolated spectral numbers for $q\in \mathbb{D}_a$, and in 
$E_{a,b}:=\mathbb{D}_a\times \mathbb{D}_b$ it has at most finitely-many 
multiple zeros. 

(3) The sequence $\rho _j$ tends to $1$ as $j$ tends to $\infty$.
\end{tm}

A statement close to part (2) is likely to appear in a text co-authored by 
the present author. The statement is written by one of his co-authors and its 
proof is obtained independently of the one of part (2) of the theorem.
 
\section{Proofs}

\begin{lm}\label{lm02}
There are no spectral numbers for $q\in \mathbb{D}_{c_0}$, 
$c_0:=0.2078750206\ldots$.
\end{lm}
(A result similar to the lemma has been independently 
announced by A. Sokal and J. Forsg{\aa}rd.)
 
\begin{proof}Indeed, set $\phi :=2\sum _{\nu =1}^{\infty}|q|^{\nu ^2/2}$. 
If $|x|=|q|^{-k-1/2}$, $k\in \mathbb{N}$, then  
in the series of $\theta$ the term 
$L:=x^kq^{k(k+1)/2}$ has the largest modulus 
(equal to $|q|^{-k^2/2}$) and the sum $S$ of the 
moduli of all other terms is  
$<|q|^{-k^2/2}\phi (|q|)$. 
One has $1\geq \phi (|q|)$ exactly when 
$|q|\leq c_0$. Hence $|L|>S$ for $|q|\leq c_0$, i.e. the term $L$ is 
{\em dominating}. Thus for no zero $\zeta$ 
of $\theta$ does one have $|\zeta |=|q|^{-k-1/2}$. For any $j$ fixed and 
for $|q|$ close to $0$ 
one has $\xi _j\sim q^{-j}$ 
(see \cite{Ko4}). Hence for 
$|q|\leq c_0$ one has 
$|q|^{-j+1/2}<|\xi _j|<|q|^{-j-1/2}$, 
i.e. all zeros are simple.
\end{proof}

\begin{prop}\label{prprpr}
For $q\in \mathbb{D}_{1/3}$ one has $|q|^{-j+1/2}<|\xi _j|<|q|^{-j-1/2}$, 
$j\geq 3$, hence the zeros $\xi _j$ of $\theta$, $j\geq 3$, 
remain simple and distinct.
\end{prop}
 
\begin{proof}
Indeed, one can suppose (see Lemma~\ref{lm02}) that $|q|\geq c_0$. 
As in the proof of that lemma  
we show that if $q\in \mathbb{D}_{1/3}$ and if 
$|x|=|q|^{-k-1/2}$, $k=2$, $3$, $\ldots$, then $\theta (q,x)\neq 0$. Set 
$x:=|q|^{-k-1/2}\omega$, $|\omega |=1$, and 
$\varepsilon :=\omega q^{k+1/2}/|q|^{k+1/2}$; thus  
$|\varepsilon |=1$. Set $\varepsilon =e^{i\beta}$ and 
$\sigma =\omega ^kq^{k(k+1)/2}|q|^{-k(k+1/2)}$. Hence    
$$\theta (q,x)=\sigma \cdot 
\left( 1+\sum _{l=1}^k(\varepsilon ^l+\varepsilon ^{-l})q^{l^2/2}+
\sum _{l=k+1}^{\infty}\varepsilon ^lq^{l^2/2}\right) =\sigma \cdot 
\left( 1+2\sum _{l=1}^k\cos (l\beta )q^{l^2/2}+
\sum _{l=k+1}^{\infty}\varepsilon ^lq^{l^2/2}\right) ~.$$
Set $\theta :=\sigma \cdot (A+B)$, where 
$A:=1+2q^{1/2}\cos \beta +2q^2\cos (2\beta )$ and 
$B:=2\sum _{l=3}^k\cos (l\beta )q^{l^2/2}+
\sum _{l=k+1}^{\infty}\varepsilon ^lq^{l^2/2}$. For $q\in \mathbb{D}_{1/3}$ one has 
$|B|<2\sum _{l=3}^{\infty}3^{-l^2/2}=
0.01456\ldots$. We show that $|A|\geq 0.0146$ for $q\in \mathbb{D}_{1/3}$. 
Set $q^{1/2}:=|q|^{1/2}e^{i\gamma}$ and

$$\begin{array}{lllll} 
R&:=&{\rm Re}(A)&=&1+2|q|^{1/2}(\cos \beta )(\cos \gamma)+2|q|^2(\cos (2\beta ))
(\cos (4\gamma ))\\ \\ 
I&:=&{\rm Im}(A)&=&2|q|^{1/2}(\cos \beta )(\sin \gamma)+
2|q|^2(\cos (2\beta ))
(\sin (4\gamma ))~.\end{array}$$ 
If $|\cos \beta |\leq 1/2$, then 
$|2q^{1/2}\cos \beta +2q^2\cos (2\beta )|\leq 1/\sqrt{3}+2/9<0.8$ and 
$|A|\geq R>0.2$. 
If $|\cos \gamma |\leq 1/2$, then again $R\geq 1-1/\sqrt{3}-2/9>0.2$ hence 
$|A|>0.2$. So suppose that $|\cos \beta |>1/2$ (hence 
$|\cos (2\beta )|<|\cos \beta |$) and $|\cos \gamma |>1/2$. 

If $(\cos \beta )(\cos \gamma )>0$, then $|R|\geq 1-2/9$ and  
$|A|>0.7$. Suppose that $(\cos \beta )(\cos \gamma )<0$. 
Set $g_0:=0.1329058248\ldots$. One has 
$|\sin (4\gamma )|\leq 4|\sin \gamma |$ and 
(as $|q|^{1/2}(1-4|q|^{3/2})$ is decreasing on $[c_0,1/3]$), 
$$|I|\geq 2|q|^{1/2}|(\sin \gamma )(\cos \beta )|(1-4|q|^{3/2})\geq 
2\cdot 3^{-1/2}(1-4/3^{3/2})|(\sin \gamma )(\cos \beta )|> 
g_0|\sin \gamma |~.$$
This is $>0.0146$ for $|\sin \gamma |>0.1098522207\ldots$, i.e. for 
$|\cos \gamma |<0.9939479310\ldots$. 

So suppose that 
$|\cos \gamma |\geq 0.9939479310\ldots$. Hence  
$\cos (4\gamma )\geq 0.904624914\ldots$ and for 
$1/2<|\cos \beta |\leq 1/\sqrt{2}$ one has $-1/2\leq \cos (2\beta )\leq 0$ and 

$$|R|\geq 1-2|q|^{1/2}/\sqrt{2}-2|q|^2\cdot (1/2)\geq 
1-(2/3)^{1/2}-(1/9)>0.07~.$$

If $1/\sqrt{2}<|\cos \beta |\leq 0.85$, then 
$0<\cos (2\beta )\leq 0.445$ and 

$$R>1-2\cdot 0.85\cdot |q|^{1/2}\geq 1-1.7/\sqrt{3}>0.018~.$$ 

If $0.85<|\cos \beta |\leq 0.93$, then 
$0.445<\cos (2\beta )\leq 0.7298$ and 

$$R>1-2\cdot 0.93\cdot |q|^{1/2}+
2\cdot |q|^2\cdot 0.904624914\ldots \cdot 0.445>0.015~.$$

If $0.93<|\cos \beta |\leq 0.98$, then 
$0.7298<\cos (2\beta )\leq 0.9208$ and 

$$R>1-2\cdot 0.98\cdot |q|^{1/2}+
2\cdot |q|^2\cdot 0.904624914\ldots \cdot 0.7298>0.015~.$$

Finally, if $0.98<|\cos \beta |\leq 1$, 
then $0.9208<\cos (2\beta )$ and 

$$R>1-2|q|^{1/2}+2\cdot |q|^2\cdot 0.904624914\ldots 
\cdot 0.9208>0.03~.$$
\end{proof}

\begin{cor}
There are no spectral numbers for $|q|\leq c_1:=0.2256613757$.
\end{cor}

\begin{proof}
Indeed, for $|q|\in J:=(0.16,c_1)$,  
$|x|=7.95$, 
in the series of $\theta$ the term $qx$ is dominating  
and $\theta \neq 0$ 
for such $(q,x)$. As $\theta (0.22,.)$ has simple zeros 
in $(-21,-19)$ and $(-7,-6)$ and $0.22^{-2.5}=44.04\ldots$, these are 
$\xi _1$ and $\xi _2$, see the proposition, so for $|q|\in J$ one has 
$|\xi _1|<7.95<|\xi _2|$ and all zeros 
of $\theta$ are simple.
\end{proof}

\begin{cor}\label{cor2}
For $q_*\in \mathbb{D}_{1/3}$, if $\theta (q_*,.)$ has a multiple zero at $x_*$, 
then $|x_*|\leq \gamma :=10.28693902\ldots$.
\end{cor}

\begin{proof}
Indeed, it is $\xi _1$ and $\xi _2$ that coalesce, see Theorem~1 of \cite{Ko2} 
(and its proof therein) and Proposition~\ref{prprpr}. 
As $\sum _{j=1}^{\infty}1/\xi _j=q$ (see \cite{Ko5}) and 
$|\xi _j|\geq |q_*|^{-j+1/2}$, $j\geq 3$, one has 
$2/|x_*|\geq f(|q_*|):=
|q_*|-\sum _{j=3}^{\infty}|q_*|^{j-1/2}=|q_*|-|q_*|^{5/2}/(1-|q_*|)$. Hence 
$|x_*|\leq 2/f(c_1)=\gamma$ ($f$ is increasing 
on $[0,0.35]$).
\end{proof}

\begin{proof}[Proof of Theorem~\ref{tmmain}.] 
The proof is based on the idea to use finite truncations of the 
series of $\theta$ as its approximations. One expects the values 
of $q$ for which these truncations have multiple zeros in $x$ to be close to 
spectral numbers of $\theta$. We use approximations which are explicit 
polynomials. In the estimations below we restrict these polynomials to 
line segments. The restrictions are polynomials in one variable whose real 
and imaginary parts are real polynomials in one variable; in the present paper  
they are of degree at most $10$. To prove that the values of 
such a polynomial $\tilde{P}$, when restricted to 
a given segment of the real axis, are smaller (or larger) 
than a given real number $\tilde{a}$, it suffices to prove that the equation 
$\tilde{P}=\tilde{a}$ has no real solution on the segment (this can be shown 
using MAPLE or Mathematica) and that for at least one point of the segment 
one has $\tilde{P}<\tilde{a}$ (or $\tilde{P}>\tilde{a}$). 

Set $U:=1+qx+q^3x^2+q^6x^3+q^{10}x^4$ and 
$V(q):=$Res$(U,U_x/q,x)/q^{26}=
256q^{10}-192q^7-128q^6+288q^5-60q^4-80q^3+52q^2-12q+1$. Set 
$K:=\{ 0\leq $Re$q,$Im$q\leq 1/3\} \supset \mathbb{D}_{1/3}$ and 
$\partial K:=$''border of $K$''. The only 
zero of $V$ in $K$ is $\lambda _0:=0.309016994374947\ldots$ 
(it is a double one).

If $\theta (q_*,.)$ has a multiple zero $x_*$, then  
$U+a_*$ and $U_x/q+b_*$ have a common zero for 

$$a_*=\varphi (q_*,x_*):=\sum _{j=5}^{\infty}q_*^{j(j+1)/2}x_*^j~~~,~~~ 
b_*=\psi (q_*,x_*):=\sum _{j=5}^{\infty}jq_*^{j(j+1)/2-1}x_*^{j-1}~.$$ 
Consider 
$\tilde{V}_{a,b}(q):=$Res$(U+a,U_x/q+b,x)/q^{26}$ for 
$|a|\leq a_0:=0.0081\ldots =\varphi (1/3,\gamma )$,
$|b|\leq b_0:=0.0119\ldots$ $=\psi (1/3,\gamma )$ 
and $|q|=1/3$, see Corollary~\ref{cor2}. One has 

$$\begin{array}{ll}
\tilde{V}_{a,b}=V+aV_1+a^2V_2+a^3V_3+b^2W_1+b^3W_2+b^4W_3+ab^2W_4~,&{\rm where}
\\ \\  
V_1=768q^{10}-384q^7-256q^6+432q^5-60q^4+34q^2-4q-80q^3~,&\\ \\  
V_2=768q^{10}-192q^7-128q^6+144q^5-27q^4~,&V_3=256q^{10}~,\\ \\  
W_1=-1-16q^5+24q^4+7q-14q^2~,&W_2=-q+4q^2-8q^4~,\\ \\  
W_3=q^4~,&W_4=6q^4-16q^5~.\end{array}$$ 

We show that $|V(q)|>|\tilde{V}_{a,b}(q)-V(q)|~(A)$ 
on $\partial K$. 
By the Rouch\'e theorem,  
$\tilde{V}_{a,b}$ has as many zeros inside $K$ as $V$, i.e. two (counted 
with multiplicity). The real and imaginary parts of the functions 
$V$, $V_1$ etc.  
(denoted by $V^R$, $V^I$ etc.) restricted to each segment 
$K_v^{\pm}:=\{ {\rm Re}q=\pm 1/3, {\rm Im}q\in [-1/3,1/3]\}$ or 
$K_h^{\pm}:=\{ {\rm Im}q=\pm 1/3, {\rm Re}q\in [-1/3,1/3]\}$ 
of $\partial K$ are 
real univariate polynomials. 

On $K_h^{\pm}$ one has 
$$\begin{array}{ll}|V^R|+|V^I|>2~~,&|V_1^R|+|V_1^I|<61.3~~,\\ \\ 
|V_k^R|+|V_k^I|<15.8~~,~~k=2,3~~,&
|W_j^R|+|W_j^I|<21.1~~,~~j=1,2,3,4,\end{array}$$
$$\begin{array}{lllllll}{\rm so}~~\Sigma &:=&
(|V^R|+|V^I|)/\sqrt{2}&-&\sum _{j=1}^3((|V_j^R|+|V_j^I|)(a_0)^j+ 
(|W_j^R|+|W_j^I|)(b_0)^{j+1})&&\\ \\ 
&&&-&(|W_4^R|+|W_4^I|)a_0(b_0)^2&>&0\end{array}$$
from which $(A)$ follows because $|V|\geq (|V^R|+|V^I|)/\sqrt{2}$, 
$|V_j|\leq |V_j^R|+|V_j^I|$ etc. On $K_v^-$ one has:
$$\begin{array}{ll}|V^R|+|V^I|>20~~,&|V_1^R|+|V_1^I|<61.3~~,\\ \\ 
|V_k^R|+|V_k^I|<15.8~~,~~k=2,3~~,&
|W_j^R|+|W_j^I|<21.5~~,~~j=1,2,3,4\end{array}$$
and again $\Sigma >0$. We consider $K_v^+\cap \{$Im$q\in [0,1/3]\}$ 
instead of $K_v^+$ (because $\theta$ is real) and 
we subdivise it into 
$K^0\cup K^1\cup K^2\cup K^3$, where 

$$\begin{array}{ll}K^0=K_v^+\cap \{$Im$q\in [0,0.05]\}~,~& 
K^1=K_v^+\cap \{$Im$q\in [0.05,0.1]\}~,\\ \\  
K^2=K_v^+\cap \{$Im$q\in [0.1,0.2]\}~,~& 
K^3=K_v^+\cap \{$Im$q\in [0.2,1/3]\}~.\end{array}$$ 
On $K^3$, $K^2$, $K^1$ and $K^0$ one has respectively 
$$\begin{array}{ll}
|V^R|+|V^I|>0.14~,~~~~~\, |V_1^R|+|V_1^I|<6.29~,&|V_k^R|+|V_k^I|<3.79~,~~~~
|W_j^R|+|W_j^I|<1.9~,\\  
|V^R|+|V^I|>0.011~,~~~~|V_1^R|+|V_1^I|<0.342~,&|V_k^R|+|V_k^I|<0.626~,~~\, 
|W_j^R|+|W_j^I|<0.47~,\\ 
|V^R|+|V^I|>0.00252~,~|V_1^R|+|V_1^I|<0.042~,&|V_k^R|+|V_k^I|<0.0865~,~
|W_j^R|+|W_j^I|<0.17~,\\ 
|V^R|+|V^I|>0.00044~,~|V_1^R|+|V_1^I|<0.0299~,&|V_k^R|+|V_k^I|<0.0464~,~
|W_j^R|+|W_j^I|<0.083~.\end{array}$$
In each case $\Sigma >0$. Hence $\tilde{V}_{a,b}$ has two zeros  
$q^{(\nu )}\in \mathbb{D}_{1/3}$, $\nu =1,2$ (counted with multiplicity). 
Consider the restrictions of $V^R$, $V^I$, $V_1^R$, $\ldots$ 
to the segment $S:=[0.29+0\cdot i,0.29+i/3]$. We set  

$$\begin{array}{ll}S=S^0\cup S^1\cup S^2\cup S^3\cup S^4~,~~{\rm where}~& 
S^0:=S\cap \{$Im$q\in [0,0.025]\}~,\\ \\  
S^1=S\cap \{$Im$q\in [0.025,0.05]\}~,& 
S^2=S\cap \{$Im$q\in [0.05,0.1]\}~,\\ \\ 
S^3=S\cap \{$Im$q\in [0.1,0.2]\}~,& 
S^4=S\cap \{$Im$q\in [0.2,1/3]\}~.\end{array}$$
On $S^4$, $S^3$, $S^2$, $S^1$ and $S^0$ one has respectively 

$$\begin{array}{ll}
|V^R|+|V^I|>0.14~,~~~\, \, |V_1^R|+|V_1^I|<5.81~,&|V_k^R|+|V_k^I|<3.72~,~~~
|W_j^R|+|W_j^I|<2.05~,\\  
|V^R|+|V^I|>0.015~,~~\, |V_1^R|+|V_1^I|<0.49~,&|V_k^R|+|V_k^I|<0.42~,~~~
|W_j^R|+|W_j^I|<0.47~,\\ 
|V^R|+|V^I|>0.0041~,~|V_1^R|+|V_1^I|<0.083~,&|V_k^R|+|V_k^I|<0.089~,~\, 
|W_j^R|+|W_j^I|<0.2~,\\ 
|V^R|+|V^I|>0.0019~,~|V_1^R|+|V_1^I|<0.0315~,~~&|V_k^R|+|V_k^I|<0.0315\, ,\, 
|W_j^R|+|W_j^I|<0.08~,\\
|V^R|+|V^I|>0.0005~,~|V_1^R|+|V_1^I|<0.024~,~~&|V_k^R|+|V_k^I|<0.015~,~\, 
|W_j^R|+|W_j^I|<0.048~.\end{array}$$
In all cases $\Sigma >0$. Hence in the rectangle 
$\{ $Re$q\in [-1/3,0.29], |$Im$q|\leq 1/3\}$ the functions $\tilde{V}_{a,b}$ 
and $V$ have the same number of zeros, i.e. none.

\begin{rem}\label{rem8841}
{\rm Knowing that there are no spectral values for $|q|\leq 0.29$ one can 
repeat the reasoning of the proof of Corollary~\ref{cor2} to show that 
$|x_*|\leq 2/f(0.29)=:\lambda =8.841250518\ldots$.}
\end{rem}

\begin{lm}\label{lmthetaxxneq0}
There are no zeros of $\theta _{xx}$ for $0<|q|\leq 0.31$ and $|x|\leq \lambda$.
\end{lm}

\begin{proof}
Indeed, $\theta _{xx}/2q^3=\sum _{j=0}^{\infty}(j+1)(j+2)q^{j(j+5)/2}x^j/2$. For 
$|q|\leq 0.31$, $|x|\leq \lambda$, the first term is dominating.
\end{proof}

\begin{lm}\label{lmXXX}
For $0<|q|\leq 0.31$ and $|x|\leq \lambda$ the equalities 
$\theta _x=\theta _q=0$ do not hold simultaneously 
hence the sets $\{ \theta =d\}$ are locally smooth.
\end{lm}

\begin{proof}
Indeed, $(\theta _q/x-\theta _x/q)/q^2x=
\sum _{j=1}^{\infty}j(j+1)q^{(j-1)(j+4)/2}x^{j-1}/2$ 
and the first term is dominating.
\end{proof}

\begin{lm}\label{lmtransversal}
For $0<|q|\leq 0.31$ and $5.946\leq |x|\leq \lambda$ one has 
$\theta _q\neq 0$.  
For $0<|q|\leq 0.31$, the curves 
$\{ \theta =d\}$ and $\{ \theta _x=0\}$ intersect transversally. 
\end{lm}

\begin{proof}
Consider the Jacobian $J:=\theta _x\theta _{qx}-\theta _{xx}\theta _q$. 
Its restriction to $\{ \theta _x=0\}$ reduces to $-\theta _{xx}\theta _q$. 
By Lemma~\ref{lmthetaxxneq0}, $\theta _{xx}\neq 0$. One has 
$\theta _q/x=1+\sum _{j=2}^{\infty}j(j+1)q^{j(j+1)/2-1}x^{j-1}/2$. We show that the 
second term is dominating for $|q|\leq 0.31$ and $5.946\leq |x|\leq \lambda$ 
which implies $\theta _q\neq 0$. Set 
$y:=q^2x$, $\chi :=\sum _{j=3}^{\infty}|j(j+1)q^{(j-1)(j-2)/2}y^{j-1}/2|$. The sum  
$1+\chi$ is maximal for $|\chi |=0.31$. For $|q|$ fixed the function 
$\chi$ is convex in $|y|$, 
so it suffices to check (for $|\chi |=0.31$) 
that $3y>1+\chi$ for $x=5.946$ and $x=\lambda$ 
which is true. For $|x|<5.946$, if $\theta _q=0$, then by Lemma~\ref{lmXXX}, 
$\theta _x\neq 0$.
\end{proof}

Lemma~\ref{lmthetaxxneq0} implies that each zero of $\theta _x$ being simple, 
it is a function in $q$ continuous in 
$\Lambda :=\mathbb{D}_{0.31}\cap \{ {\rm Re}q\geq 0.29\}$ 
and holomorphic in its interior. Consider its zero $\zeta (q)$ 
which equals $-7.5\ldots$ 
for $q=\tilde{q}$. By Lemma~\ref{lmtransversal} the intersections of the graph 
of the function $x=\zeta (q)$ with the sets $\{ \theta =d\}$ are points. 
For $d=0$ this means that there is only one point in $\Lambda$ 
belonging to the spectrum. This proves part (1) of the theorem.

Prove part (2) (from which part (3) follows). 
The set of multiple zeros of $\theta$ is a locally analytic subset in the space 
$(q,x)$. It cannot be of positive dimension. Indeed, for each $q$ fixed there 
are only isolated multiple zeros (if any) and for $0<|q|\leq 0.108$ there are 
no multiple zeros. Hence the spectrum of $\theta$ consists of isolated points. 
Further we use properties of the 
Jacobi theta function $\Theta (q,x):=\sum _{j=-\infty}^{\infty}q^{j^2}x^j$ which 
imply corollaries about the function 
$\Theta ^*(q,x)=\sum _{j=-\infty}^{\infty}q^{j(j+1)/2}x^j=
\Theta (\sqrt{q},\sqrt{q}x)$. Namely, one has 
$\Theta ^*(q,x)=qx\Theta ^*(q,qx)~(*)$.
and the zeros 
of $\Theta ^*$ are all the numbers $\mu _k:=-q^{-k}$, 
$k\in \mathbb{Z}$ (see the formula for the zeros of $\Theta$ in \cite{Wi} 
or Chapter~X of \cite{StSh}).
Further when we write $x\in \Omega _k(\delta )$ we mean 
$|x-\mu _k|\leq \delta$, 
$\delta >0$.

In a neighbourhood of  
$\mu _k$ we represent the function $\Theta ^*$ in the form
$\Theta ^*(q,x)=M(x+q^{-k})\tau (q,x)$, where $M\in \mathbb{C}\backslash 0$ 
and the series 
$\tau =1+d_1(x+q^{-k})+d_2(x+q^{-k})^2+\cdots$, $d_j\in \mathbb{C}$, 
is convergent. 
Observe that $x=x+q^{-k-s}-q^{-k-s}=-q^{-k-s}(1-q^{k+s}(x+q^{-k-s}))$ and 
$\tau (q,q^sx)=1+d_1q^s(x+q^{-k-s})+d_2q^{2s}
(x+q^{-k-s})^2+\cdots$. 
Set $X:=x+q^{-k-s}$. Using $s$ times equation $(*)$  
we represent $\Theta ^*$ close to $\mu _{k+s}$ in the form  

$$\Theta ^*(q,x)=q^{s(s+1)/2}x^sMq^sX\tau (q,q^sx)=
M(-1)^sq^{s(-2k-s+3)/2}(1-q^{k+s}X)^sX\tau (q,q^sx)~.$$

Assume that for all $s\in \mathbb{N}\cup 0$ the variable $X$ takes 
values in one and the same disk $|X|\leq \delta$. Hence 
as $s\rightarrow \infty$ 
the functions $(1-q^{k+s}X)^s$ and $\tau (q,q^sx)$ tend uniformly 
to $1$. Thus for $|X|\leq \delta$ (i.e. for $x\in \Omega _{k+s}(\delta )$) 
one has 
$\Theta ^*=M(-1)^sq^{s(-2k-s+3)/2}(1+o(1))X$.  
For any $B>0$ there exists $k_0\in \mathbb{N}\cup 0$ such that 
for $k\geq k_0$ one has $|q^{s(-2k-s+3)/2}|\geq B$ for any $s\in \mathbb{N}$. 
Hence there exists $\kappa >0$ such that for $|\eta |\leq \kappa$, 
for $k\geq k_0$ and for any 
$s\in \mathbb{N}$ the equation $\Theta ^*=\eta$ has a unique 
solution $X=X(\eta )$ with $|X|\leq \delta$.

Set $\theta (q,x)=\Theta ^*(q,x)+\Xi (q,x)$, where 
$\Xi (q,x)=-\sum _{j=-\infty}^{-1}q^{j(j+1)/2}x^j$. For any $q\in \mathbb{D}_1$ 
the series of $\Xi$ and $\Xi '$ 
converge for $|x|>1$, and for any $\varepsilon >0$ 
there exists $G\geq 1$ such that $|\Xi|\leq \varepsilon$ and 
$|\Xi '|\leq \varepsilon$ if $|x|\geq G$. (Both series (of 
$1/x$) are without constant term and the moduli of 
all coefficients of $\Xi$ are less than $1$.) This reasoning 
remains valid for $q$ replaced by a sufficiently small 
closed disk $D\subset \mathbb{D}_1$ centered at $q$.

For $k\in \mathbb{N}$ sufficiently large the equation 
$\theta (q,x)=0$, i.e. $\Theta ^*(q,x)=-\Xi (q,x)$
has a unique solution $x=x(q)\in \Omega _k(\delta )$. (For such $k$ 
and for $x\in \Omega _k(\delta )$ 
the equation $\Theta ^*(q,x)=\eta$ 
has a solution $x(q,\eta )=\mu _k+O(\eta )$ holomorphic in $\eta$  
for $|\eta |$ sufficiently small.) Substituting $-\Xi (q,x)$ for $\eta$ 
one obtains an equation 
$x=\mu _k+\Delta (q,x)$; by choosing $G$ and $k_0$ sufficiently large one 
makes $\max _{x\in \Omega _k(\delta)}|\Delta (q,x)|$ and 
$\max _{x\in \Omega _k(\delta)}|\Delta '(q,x)|$ 
arbitrarily small. 
Hence  
this equation has a unique solution in $\Omega _k(\delta)$. This reasoning is 
valid for $q\in D$.

Suppose that $a>c_0$ (for $a\leq c_0$ see Lemma~\ref{lm02}). 
The closure $\bar{\mathbb{D}}$ 
of $\mathbb{D}_a\backslash \mathbb{D}_{c_0}$ 
can be covered by finitely-many disks $D$. Hence there exists 
$j_0\in \mathbb{N}$ such that for $j\geq j_0$ and for 
$q\in \bar{\mathbb{D}}$ the function $\theta$ has a zero 
in $\Omega _j(\delta)$. Choosing, if necessary, a larger $j_0$ one can assume 
that all sets $\Omega _j(\delta)$ are disjoint (for any fixed 
$q\in \bar{\mathbb{D}}$). Hence all zeros $\xi _j$ of $\theta$ with $j\geq j_0$ 
are simple and distinct, 
so 1) one can set $b:=\min _{q\in \bar{\mathbb{D}},j\geq j_0}|\xi _j|$; 
2) one has $\rho _j\geq a$ for $j>j_0$;  
3) for $|q|<a$, $\xi _j$ depends 
meromorphically on $q$ with a pole only at $0$, i. e. its Laurent series is  
convergent in $\mathbb{D}_a$. Letting $a$ tend to $1$ one obtains the proof of 
part (3). 
\end{proof}

\end{document}